\documentclass[leqno,b5paper]{amsart}
\usepackage{epic,latexsym,amssymb}
\usepackage{color}
\usepackage{tikz}
\usepackage{amsmath,amsthm}
\usepackage{dirtytalk}
\usepackage{amsmath,amsfonts,amsthm,latexsym}
\usepackage{mathrsfs, textcomp}
\usepackage{epic,latexsym,amssymb}
\usepackage{color}
\usepackage{tikz}
\usepackage{amsmath,amsthm}
\usepackage{dirtytalk}
\usepackage{enumerate}
\newtheorem{thm}{Theorem}
\newtheorem{lem}{Lemma}

\newtheorem{cor}{Corollary}

\newtheorem{defn}{Definition}

\newcommand{\diam}{diam}

\newcommand{\Z}{{\Z B}}

\let\oldenumerate\enumerate
\renewcommand{\enumerate}{
	\oldenumerate
	\setlength{\itemsep}{0pt}
	\setlength{\parskip}{0pt}
	\setlength{\parsep}{0pt}
}

\def\vertex(#1){\put(#1){\circle*{2}}}
\def\vertexo(#1){\put(#1){\circle{2}}}
\def\vert(#1){\put(#1){\circle*{1.5}}}
\def\verto(#1){\put(#1){\circle{1.5}}}
\def\lab(#1)#2{\put(#1){\makebox(0,0)[c]{#2}}}

\begin{document}

	\title{On the generator graph of a {cyclic} group}
	
	\author{Zekhaya B. Shozi $^{*,1}$ and Teresa L. Tacbobo $^{2}$}
	
	\address{$^{1}$ School of Mathematics, Statistics \& Computer Science, University of KwaZulu-Natal, Private Bag X54001, Durban 4000, South Africa. \newline
		\indent \small {\tt Email: shoziz1@ukzn.ac.za} \newline
		\indent \small {\tt Email: zekhaya@aims.ac.za} 
	}
	
	\address{$^{2}$ Mathematics Department, College of Arts and Sciences, Bukidnon State University, Malaybalay City, Bukidnon, Philippines. \newline
		\indent \small {\tt Email: teresatacbobo@buksu.edu.ph} \newline
		\indent \small {\tt Email:  teresa\_tacbobo@yahoo.com} 
	}

	
    \thanks{$^1$ Research supported by University of KwaZulu-Natal's College of Agriculture, Engineering \& Science start-up funding for early-career researchers.}
    \thanks{$^2$ The authors gratefully acknowledge the support of the Center for Mathematical Innovations and the Research and Development Unit of Bukidnon State University.}
	
	\dedicatory{}
	\maketitle
	
	\begin{abstract}

        In this paper, we continue the study of the generator graph of a group. In 2023, Tacbobo \cite{tacbobo2023generator} defined the generator graph of a nontrivial group to be the graph whose vertices are the elements of the group, with two vertices being adjacent if at least one of them is a generator of the group. {Building on the properties} established in \cite{tacbobo2023generator}, we prove that the diameter of the generator graph of a cyclic group is at most $2$. Furthermore, we present explicit formulas for some topological indices of the generator  graph of a cyclic group with $n\ge 2$ elements and whose set of generators is $S$, {expressed} in terms of $n$ and $|S|$. Lastly, we determine the metric dimension of the generator graph of a nontrivial cyclic group {as a function of its order $n$}.

	\end{abstract}

	
	\section{Introduction}

        In graph theory, a graph can be thought of as any collection of objects some of which are related. The objects are called vertices (or nodes) and are represented by points. Two vertices are related if they are joined by a line which we call an edge. Graph theory has numerous applications in various fields, such as chemistry. For instance, in chemistry, graphs are used to study chemical molecules, where each vertex represents an atom and an edge between two vertices represents the chemical bond between the corresponding atoms.

        A group, on the other hand, can be defined as a set with an associative binary operation, a unique identity element, and unique inverses with respect to that identity element. For example, the set of integers forms a group with respect to addition. Group theory is also applicable to chemistry, especially in the study of the group of symmetries of crystals and molecules. To learn more concepts on group theory and graph theory, we refer the reader to \cite{fraleigh2003first} and \cite{henning2013total}, respectively.

        The study of groups in terms of graphs (or vice versa) have been explored by several researchers over the years. The authors in \cite{vasantha2009groups} introduced the concept of identity graph. Let $(G, \ast)$ be a group with identity element $e$, and let $\Gamma$ be the identity graph of $G$. Then the vertices of $\Gamma$ are the elements of $G$, and two vertices $x$ and $y$ of $\Gamma$ are adjacent if and only if $x \ast y = e$. The notion of the generating graph of a group was introduced in \cite{lucchini2017generating}. The generating graph $\Gamma$ of a group $G$ is a graph whose vertices are the elements of $G$, and two vertices $x$ and $y$ in $\Gamma$ if the subgroup generated by $x$ and $y$ is $G$. Erdem \cite{erdem2018generating} studied the generating graphs of symmetric groups and alternating groups, and proved the lower bound on the number of elements of a group whose generating graph is Hamiltonian. Tacbobo \cite{tacbobo2023generator} introduced the concept of the generator graph of a group. The generator graph $\Gamma(G)$ of a group $G$ is a graph whose vertex set consists of the elements of $G$, and two vertices $x$ and $y$ of $\Gamma(G)$ are adjacent if at least one of them generates $G$. Several properties of the generator graph were presented in \cite{tacbobo2023generator}. For example, the degree of each vertex in the generator graph as well as the size of the generator graph was proven in \cite{tacbobo2023generator}. Furthermore, it was shown that the generator graph of a group $G$ is the join of the complete graph whose order is the number of generators of $G$ and the null graph whose order is the difference between the number of elements of $G$ and the number of generators of $G$.

        Our aim in this paper is to provide a few more properties of the generator graph. In particular, we show that the generator graph of a group has diameter at most $2$, and that its maximum degree is at least half the number of elements of the group. Furthermore, we present explicit formulas for some topological indices of the generator graph of a group. More specifically, we present explicit formulas for the Wiener index, Gutman index, Harmonic index, Randić index and Sombor index of the generator graph of a group. We also find an explicit formula for the metric dimension of the generator graph of a group.

        The rest of the paper is structured as follows. Section \ref{preliminary} is the preliminary section where we provide and elementary review on groups and graphs. In Section \ref{Sec:some-properties-of-graphs}, we present some more properties of the generator graph of a cyclic group using the properties established in \cite{tacbobo2023generator}. Explicit formulas for some topological indices of the generator graph of a cyclic group are presented in Section \ref{Sec:some-topological-indices-of-the-generator-graph}. Lastly, in Section \ref{Sec:the-metric-dimension-of-the-generator-graph}, we present the explicit formula for the metric dimension of the generator graph of a cyclic group.

	\section{Preliminaries}
	\label{preliminary}
	
	\subsection{Basic review of groups}
		
		All group theory terminologies and notations in this work conform to the conventions established in \cite{fraleigh2003first}.
		A \emph{group} is a non-empty set $G$ together with a binary operation $\ast : G \times G \to G$ satisfying the following axioms:
		
		\begin{enumerate}
			\item Closure: For all $a, b \in G$, $a \ast b \in G$.
			\item Associativity: For all $a, b, c \in G$, $(a \ast b) \ast c = a \ast (b \ast c)$.
			\item Identity element: There exists an element $e \in G$ such that $a \ast e = e \ast a = a$ for all $a \in G$.
			\item Inverse element: For every $a \in G$, there exists $a^{-1} \in G$ such that $a \ast a^{-1} = a^{-1} \ast a = e$.
		\end{enumerate}
		
		A group $(G, \ast)$ is called an \emph{Abelian group} if, in addition to the group axioms, the binary operation is commutative:
		\[ a \ast b = b \ast a \quad \text{for all } a, b \in G.\]
		
		A \emph{subgroup} of group $G$ is a non-empty subset $H \subseteq G$ such that $H$ itself forms a group under the operation $\cdot$ restricted to $H$. Equivalently, $H$ is a subgroup if: (a) the identity element $e$ of $G$ belongs to $H$; (b) for all $a, b \in H$, $a \ast b \in H$ (\textit{closure}); and (c) for all $a \in H$, $a^{-1} \in H$ (\textit{existence of inverses}).The \emph{trivial subgroup} of $G$ is $\{ e \}$, which contains only the identity element.
		The \emph{improper subgroup} of $G$ is $G$ itself. A \emph{proper subgroup} of $G$ is any subgroup $H$ such that $H \neq G$ and $H \neq \{ e \}$.
		
		Let $G$ be a group and $g \in G$. The element $g$ is called a \emph{generator} of a subgroup $H \leq G$ if
		$H = \langle g \rangle = \{ g^n \mid n \in \mathbb{Z}$ (multiplicative notation) or
		$H=\langle g \rangle=\{ng\mid n \in \mathbb{Z}\}$ (additive notation). That is, every element of $H$ can be expressed as an integral power (or multiple) of $g$. A group $G$ is called a \emph{cyclic group} if there exists an element $g \in G$ such that 
		$G = \langle g \rangle$. In this case, $g$ is called a \emph{generator} of $G$. If $G$ is finite, the order of $G$ is equal to the order of $g$. If $G$ is infinite, it is isomorphic to $(\mathbb{Z}, +)$ in the additive case, or to the infinite cyclic group in the multiplicative case.

	\subsection{Basic review of graphs}
	
		For graph theory terminology and notation, we generally follow \cite{henning2013total}. A \emph{simple} and \emph{undirected graph} $\Gamma$ is an ordered pair $(V(\Gamma), E(\Gamma))$, where $V(\Gamma)$ is a finite nonempty set of objects called \emph{vertices} and $E(\Gamma)$ is a (possibly empty) set of $2$-element subsets of $V(\Gamma)$ called \emph{edges}. The number of vertices in a graph $\Gamma$ is the \emph{order} of $\Gamma$, and the number of edges in $\Gamma$ is the \emph{size} of $\Gamma$. If $\{u,v\} \in E(\Gamma)$, then we say $u$ and $v$ are \emph{adjacent} or are \emph{neighbours} in $\Gamma$. For simplicity, we will often write $uv$ instead of $\{u,v\}$ to represent the edge that joins the vertex $u$ and the vertex $v$ in $\Gamma$.
        
        Let $\Gamma$ be a graph with vertex set $V(\Gamma)$, edge set $E(\Gamma)$, order $n = |V(\Gamma)|$ and size $m = |E(\Gamma)|$. The \emph{open neighbourhood} of a vertex $v$ in $\Gamma$ is $N_\Gamma(v) = \{ u \in V(\Gamma)  \mid uv \in E(\Gamma)\}$, while its \emph{closed neighbourhood} is $N_\Gamma[v] = \{v\} \cup N_\Gamma(v)$. The \emph{degree} of a vertex $v$ in $\Gamma$, denoted $\deg_\Gamma (v)$, is the number of neighbours of $v$ in $\Gamma$; that is, $\deg_\Gamma (v) = |N_\Gamma (v)|$. The maximum degree in $\Gamma$ is denoted by $\Delta(\Gamma).$ 
		
		A \emph{complete graph} is a graph all whose vertices are pairwise adjacent. A complete graph of order $n$ is denoted by $K_n$.   The complement of a graph $\Gamma$, denoted $\overline{\Gamma}$, is defined as $$\overline{\Gamma}=(V(\Gamma),\{uv \mid u,v\in V(\Gamma), u\neq v \text{ and }uv\notin E(\Gamma)\}).$$
		A \emph{null graph} is a graph consisting of no edges. A null graph of order $n$ is denoted by $\overline{K_n}$. Let $\Gamma_1$ and $\Gamma_2$ be any two graphs. We define the \emph{join} of $\Gamma_1$ and $\Gamma_2$ to be the graph
		$$
		\Gamma_1\vee \Gamma_2=(V(\Gamma_1)\cup V(\Gamma_2), E(\Gamma_1)\cup E(\Gamma_2)\cup\{uv \mid u\in V(\Gamma_1)\text{ and } v\in V(\Gamma_2)\}).
		$$
        
		The \emph{distance}, $d_\Gamma (u,v)$, between two vertices $u$ and $v$ of a connected graph $\Gamma$ is the minimum number of edges in a path that connects $u$ and $v$ in $\Gamma$. The \emph{diameter} of $\Gamma$, denoted $\diam(\Gamma)$, is the distance between two vertices that are farthest from each other in $\Gamma$; that is, $\diam(\Gamma) = \max \{ d_\Gamma(u,v) \mid u,v \in V(\Gamma)\}$.
		

	\section{Some Properties of the Generator Graph}
	\label{Sec:some-properties-of-graphs}

	Several interesting properties of the generator graph of a cyclic group were established by Tacbobo in \cite{tacbobo2023generator}.   In this section, we use some of the known properties of the generator graph to establish, as corollaries, other properties of the generator graph. 

        \begin{thm}[\cite{tacbobo2023generator}]
		\label{thm:degree-of-a-vertex-in-gg}
		Let $G$ be a nontrivial group and $S$ the set of all generators of $G$. If $x \in V(\Gamma(G))$, then
		\begin{align*}
			\deg_{\Gamma(G)} (x) = \begin{cases}
				n-1 & \text{ if } x \in S,\\
				|S| & \text{ if } x \notin S.
			\end{cases}
		\end{align*}
	\end{thm}

    From Theorem \ref{thm:degree-of-a-vertex-in-gg}, we have the following corollary.
	
	\begin{cor}
		\label{cor:degree-bounds}
		Let $n \ge 3$ and $\Gamma(\mathbb{Z}_n)$ be the generator graph of $\mathbb{Z}_n$. If $x \in V(\Gamma(\mathbb{Z}_n))$, then
		\begin{align*}
			2\le \deg_{\Gamma(\mathbb{Z}_n)} (x) \le n-1.
		\end{align*}
	\end{cor} 
	
	\begin{proof}
		Let $n \ge 3$ and $\Gamma(\mathbb{Z}_n)$ be the generator graph of $\mathbb{Z}_n$. Let $S$ be the set of all generators of $\mathbb{Z}_n$. The upper bound $\deg_{\Gamma(\mathbb{Z}_n)} (x) \le n-1$ follows from Theorem \ref{thm:degree-of-a-vertex-in-gg} that if $x \in S$, then $\deg_{\Gamma(\mathbb{Z}_n)} (x) = n-1$. It remains for us to show that $\deg_{\Gamma(\mathbb{Z}_n)} (x) \ge 2$. From Theorem \ref{thm:degree-of-a-vertex-in-gg}, if $x \notin S$, then $\deg_{\Gamma(\mathbb{Z}_n)} (x) = |S|$. For a finite additive group $\mathbb{Z}_n$, $\langle x \rangle = \mathbb{Z}_n $ if and only if $\langle x^{-1}\rangle =\mathbb{Z}_n $ for every $x\in\mathbb{Z}_n $. Thus, $\{x, x^{-1}\} \subseteq S$. { When $n \ge 3$, $(n-1)^{-1}=1$ and $\langle (n-1)^{-1} \rangle = \langle1\rangle=\mathbb{Z}_n$, so $\{1,n-1\} \subseteq S$  } Therefore, $\deg_{\Gamma(\mathbb{Z}_n)}(x)=|S| \ge 2$.
	\end{proof} 

    { 
        We remark that both bounds in Corollary \ref{cor:degree-bounds} are sharp. When $n=3$, the inequality is attained at both ends. Moreover, if $n$ is prime, then $\Gamma(\mathbb{Z}_n)$ is a complete graph and hence every vertex has degree $n-1$, showing that the upper bound is always attained in this case.
    }

		\begin{thm}[\cite{tacbobo2023generator}]
			\label{thm:generator-graph-is-the-join-of-complete-graphs}
			Let $G$ be a nontrivial {cyclic} group of order $n$ and $S$ the set of all generators of $G$. If $\Gamma(G)$ is the generator graph of $G$, then $\Gamma(G) = K_{|S|} \vee\overline{ K_{n-|S|}}$.
		\end{thm}
		
		An example of the generator graph of a cyclic group is illustrated in Figure \ref{fig:An-example-of-the-generator-graph}. Note that for each vertex $u \in V(K_{|S|})$ and for each vertex $v \in V(\overline{K_{n-|S|}})$, we have $uv \in E(\Gamma(G))$. Thus, $\deg_{\Gamma(G)}(u) = n-1$ and $\deg_{\Gamma(G)} (v) = |S|$.
		
		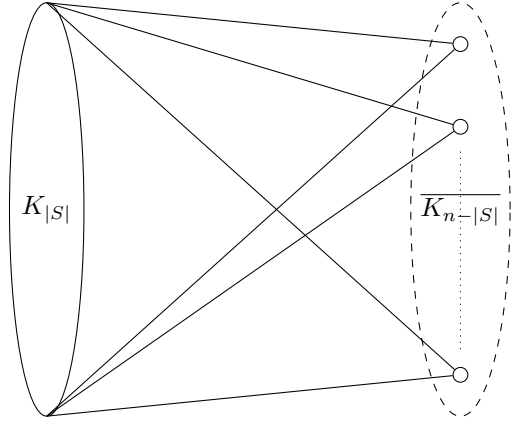
\begin{figure}[!h]
			\centering
			\begin{tikzpicture}[scale=1.1]
				\centering
				\tikzstyle{vertexX}=[circle,draw, fill=white!90, minimum size=7pt, 
				scale=0.8, inner sep=0.2pt]
				
				\node (v_1) at (0,0) [vertexX , label=above: ] {};
				\node (v_2) at (0,-1) [vertexX , label=left:] {};
				\node (v_k) at (0,-4) [vertexX , label=left: ] {};
				
				\node[fill=white] at (-5,-2) {$K_{|S|}$};
				\node[fill=white] at (0,-2) {$\overline{K_{n-|S|}}$};
				
				\draw (-5,-2) ellipse (0.45cm and 2.5cm);
				
				\draw[dashed] (0,-2) ellipse (0.6cm and 2.5cm);

				\draw (v_1) -- (-5, 0.5);
				\draw (v_1) -- (-5, -4.5);
				
				\draw (v_2) -- (-5, 0.5);
				\draw (v_2) -- (-5, -4.5);
				
				\draw (v_k) -- (-5, 0.5);
				\draw (v_k) -- (-5, -4.5);
				
				\draw [dotted] (0,-1.3) -- (0, -3.7);	
			\end{tikzpicture}
			\caption{An example of the generator graph.}
			\label{fig:An-example-of-the-generator-graph}
		\end{figure}

        We now define a special family of graphs which we call faithful graphs.

	\begin{defn}
		\label{defn:faithful-graph}
		Let $\Gamma=(V(\Gamma), E(\Gamma))$ be a graph and $e=xy \in E(\Gamma)$. If $N_\Gamma[x] \cup N_\Gamma[y] = V(\Gamma)$, then $e=xy$ is a faithful edge in $\Gamma$. If every edge of $\Gamma$ is a faithful edge, then $\Gamma$ is called a faithful graph. Let
		\begin{align*}
			\mathcal{F} = \{\Gamma \mid \Gamma \text{ is a faithful graph}\}.
		\end{align*}
	\end{defn}

        {It is clear from Definition \ref{defn:faithful-graph} that if $\Gamma \in \mathcal{F}$, then $\Gamma$ must be connected. We show, next, that every generator graph of a cyclic group is a member of the family $\mathcal{F}$. }

        \begin{thm}
		\label{thm:all-generator-graphs-are-faithful-graphs}
		Let $G$ be a nontrivial {cyclic} group. If $\Gamma(G)$ is the generator graph of $G$, then $\Gamma(G) \in \mathcal{F}$.
	\end{thm}
	
	\begin{proof}
		Let $G$ be a nontrivial {cyclic} group and $\Gamma(G)$ the generator graph of $G$. Let $e = xy \in E(\Gamma(G))$. We show that $e$ is a faithful edge in $\Gamma(G)$. Since $\Gamma(G)$ is the generator graph, at least one of the endpoints $x$ and $y$ of $e$ has degree $n-1$ in $\Gamma(G)$. Without any loss of generality, let $\deg_{\Gamma(G)}x =n-1$. This implies that $N_{\Gamma(G)}[x]=V(\Gamma(G))$. Since $\deg_{\Gamma(G)}y \le n-1$, we have $N_{\Gamma(G)}[y]\subseteq N_{\Gamma(G)}[x]=V(\Gamma(G))$. Thus,
		$N_{\Gamma(G)}[x] \cup N_{\Gamma(G)}[y] =V(\Gamma(G))$, and hence $e$ is a {faithful} edge in $\Gamma(G)$. Thus, $\Gamma(G) \in \mathcal{F}$.
	\end{proof}
	
	Since we have now shown that every generator graph of a cyclic group belongs to $\mathcal{F}$, we will, going forward, prove the properties of $\mathcal{F}$, so that the properties of the generator graph of a cyclic group will follow immediately.
	
	\begin{thm}
		\label{thm:G-in-F-implies-diam-G-is-2}
		If $\Gamma\in \mathcal{F}$, then  $\diam(\Gamma)\le 2$.
	\end{thm}
	
	\begin{proof}
		Let $\Gamma\in \mathcal{F}$ and $e=uv\in E(\Gamma)$. Since $\Gamma\in \mathcal{F}$, $N_\Gamma[u]\cup N_\Gamma[v]=V(\Gamma)$. Now let $u^*, v^*\in V(\Gamma)$. Since $e=uv$ is a faithful edge in $\Gamma$, we have $u^*, v^* \in N_\Gamma[u]\cup N_\Gamma[v]$. If $u^*, v^* \in N_\Gamma[u]$, then $d_\Gamma(u^*, v^*)\le 2$, implying that $\diam(\Gamma)\le 2$. Similarly, if $u^*, v^* \in N_\Gamma[v]$, then $d_\Gamma(u^*, v^*)\le 2$, implying that $\diam(\Gamma)\le 2$. Hence, we may assume that $u^* \in N_\Gamma[u]$ and $v^* \in N_\Gamma[v]$. If $u^*v^* \in E(\Gamma)$, then $d_\Gamma(u^*, v^*)=1\le 2$, so the conclusion $\diam(\Gamma)\le 2$ holds. Hence, we may assume that $u^*v^*\notin E(\Gamma)$. Consider the edge $uu^* \in E(\Gamma)$. Note that the vertex $v^*\in N_\Gamma[v]$ is adjacent to neither $u$ nor $u^*$; that is, $v^* \notin N_\Gamma[u]$ and $v^* \notin N_\Gamma[u^*]$. Thus, $v^* \notin N_\Gamma[u]\cup N_\Gamma[u^*]$, implying that $ N_\Gamma[u]\cup N_\Gamma[u^*] \ne V(\Gamma)$. This, however, contradicts our assumption that $\Gamma \in \mathcal{F}$.
	\end{proof}
	
	We remark that the converse of Theorem \ref{thm:G-in-F-implies-diam-G-is-2} is not true in general. As a counterexample, consider the cycle $C_5$ as shown in Figure \ref{fig:The-cycle-C5}. Note that 
	\begin{align*}
		\diam(C_5) = \left \lfloor \frac{5}{2} \right \rfloor = 2.
	\end{align*}
	However, $N_{C_5}[u] \cup N_{C_5}[v] = V(C_5)\setminus \{x\} \ne V(C_5)$, so  $uv \in E(C_5)$ is not a faithful edge. Thus, $C_5 \notin \mathcal{F}$. In fact, none of the edges of $C_5$ is a faithful edge.
	
	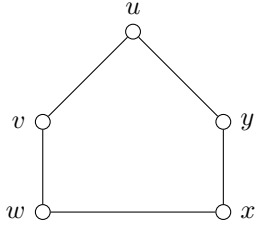
\begin{figure}[!h]
		\centering
		
		\begin{tikzpicture}[scale=1.2]
			\centering
			\tikzstyle{vertexX}=[circle,draw, fill=white!90, minimum size=7pt, 
			scale=0.8, inner sep=0.2pt]
			
			\node (u) at (0,0) [vertexX , label=above:$u$ ] {};
			\node (v) at (-1,-1) [vertexX , label=left:$v$] {};
			\node (w) at (-1,-2) [vertexX , label=left: $w$] {};
			\node (x) at (1,-2) [vertexX , label=right:$x$] {};
			\node (y) at (1,-1) [vertexX , label=right:$y$] {};
			
			
			\draw (u) -- (v) -- (w) -- (x) -- (y) -- (u) -- cycle;
			
		\end{tikzpicture}
		\caption{The cycle $C_5$.}
		\label{fig:The-cycle-C5}
	\end{figure}

        As a consequence of Theorem \ref{thm:all-generator-graphs-are-faithful-graphs} and  Theorem \ref{thm:G-in-F-implies-diam-G-is-2}, we obtain the following corollary.
	\begin{cor}
		\label{cor:diameter-of-the-generator-graph}
		If $\Gamma(G)$ is the generator graph of a finite cyclic group $G$, then $\diam(\Gamma(G)) \le 2$.
	\end{cor}

  \begin{thm}[\cite{tacbobo2023generator}]
		\label{thm:group-of-prime-order-in-gg}
		Let $G$ be a nontrivial group of order $n$. The generator graph $\Gamma(G)$ of $G$ is complete if and only if $n$ is prime.
	   \end{thm}
       
    {

    The exact value for the diameter of the generator graph of a nontrivial finite cyclic group is presented in the next result as a consequence of  Corollary \ref{cor:diameter-of-the-generator-graph} and Theorem \ref{thm:group-of-prime-order-in-gg}.
    
       \begin{cor}
		\label{cor:the-diameter-of-the-generator-graph-exact}
		Let $G$ be a nontrivial cyclic group of order $n$, and $\Gamma(G)$  the generator graph of $G$. Then the diameter of $\Gamma(G)$ is given by
			\begin{align*}
				\diam(\Gamma(G)) = \begin{cases}
					1 &\text{ if } n \text{ is prime},\\
					2 & \text { otherwise}.
				\end{cases}
			\end{align*}
	   \end{cor} 
    }
	
	\begin{thm}
		\label{thm:the-max-degree-of-a-vertex-of-Gamma-in-F-is-atleast-n-over-2}
		If $\Gamma \in \mathcal{F}$, then $\Delta \ge \displaystyle \frac{|V(\Gamma)|}{2}$.
	\end{thm}
	
	\begin{proof}
		Suppose, to the contrary, that $\Gamma \in \mathcal{F}$ and $\Delta(\Gamma) \le \displaystyle \frac{|V(\Gamma)|}{2} -1$. Let $v\in V(\Gamma)$ such that $\deg_\Gamma v = \Delta(\Gamma)$. Observe that $|N_\Gamma(v)| + |\overline{N_\Gamma(v)}| +1 = |V(\Gamma)|$. Thus,
		\begin{align*}
			|\overline{N_\Gamma(v)}| &= |V(\Gamma)| - |N_\Gamma(v)| -1\\
			&\ge |V(\Gamma)| - \left( \frac{|V(\Gamma)|}{2}  -1 \right) -1\\
			&=\frac{|V(\Gamma)|}{2}.
		\end{align*}
		
		Let $e=uv \in E(\Gamma)$ for some $u \in N_\Gamma (v)$. Since $\Delta(\Gamma) \le \displaystyle \frac{|V(\Gamma)|}{2} -1$, every vertex $u \in N_\Gamma (v)$ is adjacent to at most $\displaystyle \frac{|V(\Gamma)|}{2} -2$ vertices of $\overline{N_\Gamma (v)}$. Thus, since $|\overline{N_\Gamma (v)}| \ge \displaystyle \frac{|V(\Gamma)|}{2}$, there exists a vertex $u^* \in \overline{N_\Gamma (v)}$ such that $uu^* \notin E(\Gamma)$. Since $vu^* \notin E(\Gamma)$, we have $u^* \notin N_\Gamma [u] \cup N_\Gamma [v]$, and so $N_\Gamma [u] \cup N_\Gamma [v] \ne V(\Gamma)$. Thus, $e=uv \in E(\Gamma)$ is not a faithful edge, contradicting our supposition that $\Gamma \in \mathcal{F}$.
	\end{proof}

	The result of Theorem \ref{thm:the-max-degree-of-a-vertex-of-Gamma-in-F-is-atleast-n-over-2} is a necessary condition for a graph $\Gamma$ to be an element of $\mathcal{F}$ but it is not sufficient. To illustrate this, consider the graph $\Gamma$ shown in Figure \ref{fig:a-non-faithful-graph}. Note that $\Delta(\Gamma) = 3 \ge 2 = \displaystyle \frac{|V(\Gamma)|}{2}$. However, the edge $e=uv \in E(\Gamma)$ is not faithful since $N_\Gamma[u] \cup N_\Gamma[v] = V(\Gamma)\setminus \{w\} \ne V(G)$. Thus, $\Gamma \notin \mathcal{F}$.
	
	\begin{figure}[!h]
		\centering
		
		\begin{tikzpicture}[scale=1.2]
			\centering
			\tikzstyle{vertexX}=[circle,draw, fill=white!90, minimum size=7pt, 
			scale=0.8, inner sep=0.2pt]
			
			\node (u) at (0,0) [vertexX , label=right:$u$ ] {};
			\node (v) at (-2,0) [vertexX , label=left:$v$] {};
			\node (w) at (-2,-2) [vertexX , label=left: $w$] {};
			\node (x) at (0,-2) [vertexX , label=right:$x$] {};
			
			
			\draw (u) -- (v) --  (x) -- (u) -- cycle;
			\draw (w) -- (x);
			
		\end{tikzpicture}
		\caption{A graph $\Gamma$.}
		\label{fig:a-non-faithful-graph}
	\end{figure}
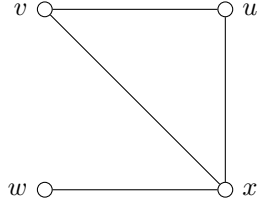

        As a consequence of Theorem \ref{thm:all-generator-graphs-are-faithful-graphs} and Theorem \ref{thm:the-max-degree-of-a-vertex-of-Gamma-in-F-is-atleast-n-over-2}, we obtain the following corollary.
	\begin{cor}
		\label{cor:minimum-maximum-degree-of-th-generator-graph}
		Let $G$ be a nontrivial {cyclic} group of order $n$. If $\Gamma(G)$ is the generator graph of $G$, then $$\Delta(\Gamma(G)) \ge \displaystyle \frac{|V(\Gamma(G))|}{2} = \frac{n}{2}.$$
	\end{cor}

        We remark that the result of Corollary \ref{cor:minimum-maximum-degree-of-th-generator-graph} is weak because we already know that if $\Gamma(G)$ is the generator graph of a nontrivial cyclic group $G$ of order $n$, then $$\Delta(\Gamma(G)) = n-1 > \frac{n}{2}.$$

		\section{Some Topological Indices of the Generator Graph}
        \label{Sec:some-topological-indices-of-the-generator-graph}

		In this section, we study some of the topological indices of the generator graph of a nontrivial cyclic group. In particular, we present explicit formulas for the Wiener index, the Gutman index, the harmonic index, the Randić index, and the Sombor index of the generator graph $\Gamma(G)$ of a cyclic group $G$ of order $n\ge 2$, whose set of generators is denoted by $S$. Each of these explicit formulas is presented in terms of the group's order $n$ and the number of the group's generators $|S|$. { Moreover, for each topological index considered, we also provide the explicit formula for the special case where $n$ is prime. }
		
		\subsection{The Wiener Index of the Generator Graph} 
		
		In this subsection, we study the Wiener index of the generator graph of a nontrivial cyclic group. The study of the Wiener index goes back to 1947. This index was first introduced by a chemist, Herold Wiener, hence it is called the Wiener index. In \cite{wiener1947structural}, Wiener  showed that there is a correlation between the Wiener index and boiling points of saturated hydrocarbons. Mathematicians now study this index beyond the class of chemical graphs.
		
		Let $\Gamma$ be a connected graph. The \emph{Wiener index} of $\Gamma$, denoted $W(\Gamma)$, is sum of the distances between all unordered pairs of vertices of $\Gamma$, i.e.,
		\begin{align*}
			W(\Gamma) = \displaystyle \sum\limits_{ \{u,v\} \subseteq V(\Gamma) } d_\Gamma (u,v).
		\end{align*}

	Using Theorem \ref{thm:generator-graph-is-the-join-of-complete-graphs}, we present the explicit formula for the Wiener index of the generator graph $\Gamma(G)$ of any cyclic group $G$ with {$n\ge 2$} elements and $|S| \ge 2$ generators.
	
	\begin{thm}
		\label{thm:the-wiener-index-of-Gamma-G}
		Let $G$ be a nontrivial cyclic group of order $n$, and $S$ the set of all generators of $G$. If $\Gamma(G)$ is the generator graph of $G$, then the Wiener index of $\Gamma(G)$ is given by
		\begin{align*}
			W(\Gamma(G)) = \frac{1}{2}|S|^2 - \frac{1}{2}(2n-1)|S| + n^2-n.
		\end{align*}
	\end{thm}
	
	\begin{proof}
		Let $\Gamma(G)$ be the generator graph of a nontrivial cyclic group of order $n$ whose set of generators is denoted by $S$. For an integer $i\ge 1$, let
		\begin{align*}
			D_i(u,v) = \{ \{u,v\} \subseteq V(\Gamma(G)) \mid d_{\Gamma(G)}(u,v) = i\}.
		\end{align*}
		By Corollary \ref{cor:diameter-of-the-generator-graph}, $diam(\Gamma(G))\le 2$, so { $\displaystyle \bigcup\limits_{u,v \in V(\Gamma(G)) } \{u,v\} = V(\Gamma(G))$ for all $ \{u,v\} \in \displaystyle \bigcup\limits_{i=1}^2 D_i(u,v)$. } 
        By Theorem \ref{thm:generator-graph-is-the-join-of-complete-graphs}, $\Gamma(G) = K_{|S|} \vee\overline{ K_{n-|S|}}$. Thus, 
		\begin{align*}
			|D_1(u,v)| = \binom{|S|}{2} + (n-|S|)|S|
		\end{align*}
		and 
		\begin{align*}
			|D_2(u,v)| = \binom{n-|S|}{2}.
		\end{align*}
		The Wiener index of $\Gamma(G)$ is
		\begin{align*}
			W(\Gamma(G)) &={ \displaystyle \sum\limits_{ \{u,v\} \subseteq V(\Gamma(G)) } d_{\Gamma(G)} (u,v)}\\
			&={ \displaystyle \sum\limits_{ \{u,v\} \in D_1(u,v) } d_{\Gamma(G)} (u,v) + \displaystyle \sum\limits_{ \{u,v\} \in D_2(u,v) } d_{\Gamma(G)} (u,v)  } \\
			&= 1\cdot|D_1(u,v)| + 	2\cdot|D_2(u,v)|\\
			&=\left[\binom{|S|}{2} + (n-|S|)|S|\right] + 2\cdot\left[\binom{n-|S|}{2}\right] \\
			&=\frac{|S|(|S|-1)}{2} + (n-|S|)|S| + 2\cdot\left[\frac{(n-|S|)(n-|S|-1)}{2}\right]\\
			&= \frac{1}{2}|S|^2 - \frac{1}{2}|S| + (n-|S|)(n-1)\\
			&=\frac{1}{2}|S|^2 - \frac{1}{2}|S| + n^2 - n -n|S| + |S|\\
			&=\frac{1}{2}|S|^2 -\frac{1}{2}(2n-1)|S| + n^2-n,
		\end{align*}
		as desired.
	\end{proof}

         The following result follows from Theorem \ref{thm:the-wiener-index-of-Gamma-G} and Theorem \ref{thm:group-of-prime-order-in-gg}.  Note that Theorem \ref{thm:group-of-prime-order-in-gg} implies that if $n$ is prime, then every non-identity element of a group $G$ is a generator of $G$. Thus, if $S$ is the set of all generators of $G$, then $n$ can be expressed as $n=1+|S|$.

      \begin{cor}
		\label{cor:the-wiener-index-of-Gamma-G-of-prime-order}
		Let $G$ be a nontrivial cyclic group of prime order $p$. If $\Gamma(G)$ is the generator graph of $G$, then the Wiener index of $\Gamma(G)$ is given by
		\begin{align*}
			W(\Gamma(G)) = \frac{1}{2}p(p-1).
		\end{align*}
	\end{cor}

	\subsection{The Gutman Index of the Generator Graph} 
		
		The Gutman index, {introduced in \cite{Mazorodze2014GutmanIndex}}, can be thought of as an extension of the Wiener index. It is one of the topological indices that are both distance-based and degree-based.  Let $\Gamma$ be a connected graph. The \emph{Gutman index} of $\Gamma$, denoted $Gut(\Gamma)$, is  
		\begin{align*}
			Gut(\Gamma) = \displaystyle \sum\limits_{ \{u,v\} \subseteq V(\Gamma) } \deg_\Gamma (u) \deg_\Gamma (v) d_\Gamma (u,v).
		\end{align*}
		
		In this subsection, we present the explicit formula for the Gutman index of the generator graph $\Gamma(G)$ of any nontrivial cyclic group $G$ with $|S| \ge 2$ generators.

			\begin{thm}
			\label{thm:the-gutman-index-of-Gamma-G}
			Let $G$ be a nontrivial {cyclic} group of order $n$, and $S$ the set of all generators of $G$. If $\Gamma(G)$ is the generator graph of $G$, then the Gutman index of $\Gamma(G)$ is given by
			\begin{align*}
				Gut(\Gamma(G)) = \frac{1}{2}|S|(|S|-1)(n-1)^2 + |S|^2(n-|S|)(2n-|S|-2).
			\end{align*}
		\end{thm}
		
		\begin{proof}
			Let $\Gamma(G)$ be the generator graph of a nontrivial {cyclic} group of order $n$ whose set of generators is denoted by $S$. By Theorem \ref{thm:generator-graph-is-the-join-of-complete-graphs}, $\Gamma(G) = K_{|S|} \vee\overline{ K_{n-|S|}}$, and by Theorem \ref{thm:degree-of-a-vertex-in-gg}, 
			\begin{align*}
				\deg_{\Gamma(G)} (x) = \begin{cases}
					n-1 & \text{ if } x \in S,\\
					|S| & \text{ if } x \notin S.
				\end{cases}
			\end{align*}
			The Gutman index of $\Gamma(G)$ is
			\begin{align*}
				Gut(\Gamma(G)) =& \sum\limits_{\{x,y\}\subseteq{V(\Gamma(G)}}\deg_{\Gamma(G)}(x)\deg_{\Gamma(G)}(y)d_{\Gamma(G)}(x,y) \\
				=&\sum\limits_{\{x,y\}\in{D_1(x,y)}}\deg_{\Gamma(G)}(x)\deg_{\Gamma(G)}(y)d_{\Gamma(G)}(x,y) \\
				&+ \sum\limits_{\{x,y\}\in{D_2(x,y)}}\deg_{\Gamma(G)}(x)\deg_{\Gamma(G)}(y)d_{\Gamma(G)}(x,y) \\
				=&\sum\limits_{\substack{ \{x,y\} \in D_1(x,y) \\ x,y \in S}}\deg_{\Gamma(G)}(x)\deg_{\Gamma(G)}(y)d_{\Gamma(G)}(x,y)\\
				& + \sum\limits_{\substack{ \{x,y\} \in D_1(x,y) \\ x \in S \\ y \in V(\Gamma)\setminus S  }}\deg_{\Gamma(G)}(x)\deg_{\Gamma(G)}(y)d_{\Gamma(G)}(x,y)\\\\
				& + \sum\limits_{\substack{ \{x,y\} \in D_2(x,y) \\ x,y \in V(\Gamma) \setminus S }}\deg_{\Gamma(G)}(x)\deg_{\Gamma(G)}(y)d_{\Gamma(G)}(x,y)\\ \\
				=&\binom{|S|}{2}(n-1)^2 +|S|^2(n-|S|)(n-1) + \binom{n-|S|}{2}|S|^2\cdot2\\
				=&\frac{|S|(|S|-1)}{2}(n-1)^2 + |S|^2(n-|S|)(n-1) +|S|^2 (n-|S|)(n-|S|-1)\\
				=&\frac{1}{2}|S|(|S|-1)(n-1)^2 + |S|^2(n-|S|)(2n-|S|-2),
			\end{align*}
			as stated.
		\end{proof}

            The following result is a special case of Theorem \ref{thm:the-gutman-index-of-Gamma-G}, obtained when the order of the group $G$ is prime.
        \begin{cor}
			\label{cor:the-gutman-index-of-Gamma-G-of-prime-order}
			Let $G$ be a nontrivial {cyclic} group of prime order $p$. If $\Gamma(G)$ is the generator graph of $G$, then the Gutman index of $\Gamma(G)$ is given by
			\begin{align*}
				Gut(\Gamma(G)) = \frac{1}{2}p(p-1)^3.
			\end{align*}
		\end{cor}

	\subsection{The Harmonic Index of the Generator Graph}
		
		In this section we present the explicit formula for the harmonic index {~\cite{Zhong2012Harmonic}} of the generator graph $\Gamma(G)$ of a nontrivial cyclic group $G$ with $|S| \ge 2$ generators.  {Let $\Gamma$ be a connected graph with vertex set $V(\Gamma)$ and edge set $E(\Gamma)$. 
			The \emph{harmonic index} of $\Gamma$, denoted by $H(\Gamma)$, is defined as
			\[ H(\Gamma) = \sum_{uv \in E(G)} \frac{2}{\deg_\Gamma(u) + \deg_\Gamma(v)},\]
			where $\deg_\Gamma(u)$ and $\deg_\Gamma(v)$ are the degrees of the vertices $u$ and $v$, respectively.}

		\begin{thm}
			\label{thm:the-harmonic-index-of-Gamma-G}
			Let $G$ be a nontrivial cyclic group of order $n$, and $S$ the set of all generators of $G$. If $\Gamma(G)$ is the generator graph of $G$, then the Harmonic index of $\Gamma(G)$ is given by
			\begin{align*}
				H(\Gamma(G)) = \frac{|S|(|S|-1)}{2(n-1)} + (n-|S|) \left[ \frac{(n-1)^2 + 3|S|^2}{2|S|(n+|S|-1)} \right].
			\end{align*}
		\end{thm}

		\begin{proof}
			Let $\Gamma(G)$ be the generator graph of $G$ of a nontrivial cyclic group of order $n$, and let $S$ be the set of all generators of $G$. By Theorem \ref{thm:degree-of-a-vertex-in-gg}, 
			\begin{align*}
				\deg_{\Gamma(G)} (x) = \begin{cases}
					n-1 & \text{ if } x \in S,\\
					|S| & \text{ if } x \notin S,
				\end{cases}
			\end{align*}
            and by Theorem \ref{thm:generator-graph-is-the-join-of-complete-graphs}, $\Gamma(G) = K_{|S|} \vee\overline{ K_{n-|S|}}$.
            Thus, the harmonic index of $\Gamma(G)$ is 
			\begin{align*}
				H(\Gamma(G)) =& \sum\limits_{u,v \in V(\Gamma(G))} \frac{2}{\deg_\Gamma (u) + \deg_\Gamma (v)}\\
				=& \sum\limits_{u,v \in S} \frac{2}{\deg_\Gamma (u) + \deg_\Gamma (v)} + \sum\limits_{u,v \in V(\Gamma(G))\setminus S } \frac{2}{\deg_\Gamma (u) + \deg_\Gamma (v)}\\
				&+\sum\limits_{\substack{u \in S  \\ v \in V(\Gamma(G))\setminus S }} \frac{2}{\deg_\Gamma (u) + \deg_\Gamma (v)}\\
				=&\binom{|S|}{2} \left[\frac{2}{n-1 + n-1} \right] + \binom{n-|S|}{2} \left[\frac{2}{|S| + |S|} \right] + |S|(n-|S|) \left[\frac{2}{n-1+|S|}\right]\\
				=&\frac{|S|(|S|-1)}{2(n-1)} + \frac{(n-|S|)(n-|S|-1)}{2|S|} + \frac{2|S|(n-|S|)}{n+|S|-1}\\
				=&\frac{|S|(|S|-1)}{2(n-1)} + (n-|S|) \left[ \frac{n-|S|-1}{2|S|} + \frac{2|S|}{n+|S|-1} \right]\\
				=& \frac{|S|(|S|-1)}{2(n-1)} + (n-|S|) \left[ \frac{(n-|S|-1)(n+|S|-1) + 4|S|^2}{2|S|(n+|S|-1)} \right]\\
				=& \frac{|S|(|S|-1)}{2(n-1)} + (n-|S|) \left[ \frac{(n-1)^2 + 3|S|^2}{2|S|(n+|S|-1)} \right],
			\end{align*}
			as stated.
		\end{proof}
	
            We now present a corollary, which follows from Theorem \ref{thm:the-harmonic-index-of-Gamma-G} by considering the case where the order of group $G$ is prime.

        \begin{cor}
			\label{cor:the-harmonic-index-of-Gamma-G-of-prime-order}
			Let $G$ be a nontrivial {cyclic} group of prime order $p$. If $\Gamma(G)$ is the generator graph of $G$, then the Harmonic index of $\Gamma(G)$ is given by
			\begin{align*}
				H(\Gamma(G)) = \frac{p}{2}.
			\end{align*}
		\end{cor}
        
		\subsection{The Randić Index of the Generator Graph.}
		In this subsection we present the explicit formula for the Randić index ~\cite{randic1975characterization} of the generator graph $\Gamma(G)$ of a nontrivial cyclic group $G$ with $|S| \ge 2$ generators. The \emph{Randić Index} is a well-known topological index in mathematical chemistry, introduced by Milan Randić in 1975. It is defined for a simple connected graph {$\Gamma = (V(\Gamma), E(\Gamma))$} as
		\[ R(\Gamma) = \sum_{uv \in E(\Gamma)} \frac{1}{\sqrt{\deg_\Gamma(u) \cdot \deg_\Gamma(v)}},\] 
		where $\deg_\Gamma(u)$ and $\deg_\Gamma(v)$ denote the degrees of the vertices $u$ and $v$ in $\Gamma$, {respectively}. The Randić Index has been widely used in the study of molecular branching and quantitative structure–activity relationships (QSAR).
		
		\begin{thm}
			\label{thm:the-randić-index-of-Gamma-G}
			Let $G$ be a nontrivial cyclic group of order $n$, and $S$ the set of all generators of $G$. If $\Gamma(G)$ is the generator graph of $G$, then the Randić index of $\Gamma(G)$ is given by
			\begin{align*}
				R(\Gamma(G)) = \frac{1}{n-1}\left(\frac{1}{2}|S|(|S|-1) +  (n-|S|)\sqrt{(n-1)|S|}\right).
			\end{align*}
		\end{thm}
		
		\begin{proof}
			Let $\Gamma(G)$ be the generator graph of a nontrivial {cyclic} group of order $n$ whose set of generators is denoted by $S$. By Theorem \ref{thm:degree-of-a-vertex-in-gg},  if $x\in S$, then $\deg_{\Gamma(G)} (x) = n-1$,  and if { $x \in V(\Gamma(G)) \setminus S$ }, then $\deg_{\Gamma(G)} (x) = |S|$.
			
			The Randić index of $\Gamma(G)$ is
			\begin{align*} 
				R(\Gamma(G)) =& \sum\limits_{uv\in E(\Gamma(G)}\frac{1}{\sqrt{\deg_{\Gamma(G)}(u)\deg_{\Gamma(G)}(v)}}\\
				=& \sum\limits_{\substack{uv \in E(\Gamma(G))  \\ u,v \in S}} \frac{1}{\sqrt{\deg_{\Gamma(G)}(u)\deg_{\Gamma(G)}(v)}} + \sum\limits_{\substack{uv \in E(\Gamma(G))  \\ u \in S \\ v \in V(\Gamma(G)) \setminus S}} \frac{1}{\sqrt{\deg_{\Gamma(G)}(u)\deg_{\Gamma(G)}(v)}} \\  
				=& \frac{|S|(|S|-1)}{2 \sqrt{(n-1)^2}} + \frac{|S|(|n|-|S|)}{\sqrt{(n-1)|S|}} \\
				=& \frac{|S|(|S|-1)}{2 (n-1)} + \frac{|S|(|n|-|S|)}{\sqrt{(n-1)|S|}}\\
				=& \frac{|S|(|S|-1)}{2 (n-1)} + \frac{(|n|-|S|)\sqrt{(n-1)|S|}}{n-1}\\
				=& \frac{1}{n-1}\left(\frac{1}{2}|S|(|S|-1) +  (n-|S|)\sqrt{(n-1)|S|}\right),
			\end{align*}
			as stated.
		\end{proof}

            The next result is a direct consequence of Theorem \ref{thm:the-randić-index-of-Gamma-G} for the case where $n$ is prime.

            \begin{cor}
			\label{cor:the-randić-index-of-Gamma-G-of-prime-order}
			Let $G$ be a nontrivial {cyclic} group of prime order $p$. If $\Gamma(G)$ is the generator graph of $G$, then the Randić index of $\Gamma(G)$ is given by
			\begin{align*}
				R(\Gamma(G)) = \frac{p}{2}.
			\end{align*}
		\end{cor}

		\subsection{The Sombor Index of the Generator Graph.}
		
		In this subsection we present the explicit formula for the Sombor index ~\cite{gutman2021geometric} of the generator graph $\Gamma(G)$ of a nontrivial cyclic group $G$ with $|S| \ge 2$ generators. The Sombor index is one of the degree-based topological indices in mathematical chemistry, and it was introduced in 2021 by Ivan Gutman \cite{gutman2021geometric}. Let $\Gamma = (V(\Gamma), E(\Gamma))$ be a simple connected graph. The \emph{Sombor index} of $\Gamma$, denoted $SO(\Gamma)$, is defined as
		\begin{align*}
			SO(\Gamma) = \displaystyle \sum\limits_{uv \in E(\Gamma)} \sqrt{(\deg_\Gamma (u))^2 + (\deg_\Gamma (v))^2},
		\end{align*}
		where $\deg_{\Gamma} (u)$ and $\deg_\Gamma (v)$ denote the degrees of the vertices $u$ in $\Gamma$ and $v$ in $\Gamma$, respectively.

		\begin{thm}
			\label{thm:the-sombor-index-of-Gamma-G}
			Let $G$ be a nontrivial cyclic group of order $n$, and $S$ the set of all generators of $G$. If $\Gamma(G)$ is the generator graph of $G$, then the Sombor index of $\Gamma(G)$ is given by
			\begin{align*}
				SO(\Gamma(G))&=\frac{\sqrt{2}}{2}|S|(|S|-1)(n-1) + |S|(n-|S|) \sqrt{(n-1)^2+|S|^2}.
			\end{align*}
		\end{thm}
		
		\begin{proof}
			Let $\Gamma(G)$ be the generator graph of a nontrivial cyclic group of order $n$ whose set of generators is denoted by $S$. By Theorem \ref{thm:degree-of-a-vertex-in-gg},  if $v\in S$, then $\deg_{\Gamma(G)} (u) = n-1$,  and if $v \in V(\Gamma(G)) \setminus S$, then $\deg_{\Gamma(G)} (v) = |S|$. The Sombor index of $\Gamma(G)$ is
			\begin{align*}
				SO(\Gamma) &= \displaystyle \sum\limits_{uv \in E(\Gamma)} \sqrt{(\deg_{\Gamma(G)} (u))^2 + (\deg_{\Gamma(G)} (v))^2}\\
				&=\sum\limits_{\substack{uv \in E(\Gamma(G))  \\ u,v \in S}} \sqrt{(\deg_{\Gamma(G)} (u))^2 + (\deg_{\Gamma(G)} (v))^2} + \sum\limits_{\substack{uv \in E(\Gamma(G))  \\ u \in S \\ v\in V(\Gamma(G))  \setminus S}} \sqrt{(\deg_{\Gamma(G)} (u))^2 + (\deg_{\Gamma(G)} (v))^2}\\
				&=\binom{|S|}{2} \sqrt{2(n-1)^2} + |S|(n-|S|) \sqrt{(n-1)^2+|S|^2}\\
				&=\frac{|S|(|S| - 1)}{2}\cdot \sqrt{2}(n-1) + |S|(n-|S|) \sqrt{(n-1)^2+|S|^2}\\
				&=\frac{\sqrt{2}}{2}|S|(|S|-1)(n-1) + |S|(n-|S|) \sqrt{(n-1)^2+|S|^2},
			\end{align*}
			as stated.
		\end{proof}
		
		       We now present a corollary, which follows from Theorem \ref{thm:the-sombor-index-of-Gamma-G} by considering the case where the order of group $G$ is prime.

            \begin{cor}
			\label{cor:the-sombor-index-of-Gamma-G-of-prime-order}
			Let $G$ be a nontrivial {cyclic} group of prime order $p$. If $\Gamma(G)$ is the generator graph of $G$, then the Sombor index of $\Gamma(G)$ is given by
			\begin{align*}
				SO(\Gamma(G))&=\frac{\sqrt{2}}{2}p(p-1)^2.
			\end{align*}
		\end{cor}

		\section{The Metric Dimension of the Generator Graph}
        \label{Sec:the-metric-dimension-of-the-generator-graph}
		

        In this section, we investigate the metric dimension of the generator graph $\Gamma(G)$ of a finite cyclic group $G$ of order $n$. The metric dimension of a graph is a fundamental parameter that quantifies its ability to uniquely distinguish all vertices based on their distances to a selected set of reference vertices, called a resolving set. A subset $W \subseteq V(\Gamma)$ of a graph $\Gamma$ is a called \emph{resolving set} of $\Gamma$ if, for any two distinct vertices $u, v \in V(\Gamma)$, their distance representations with respect to $W$ are different; that is,
        \[ r(u \mid W) = \big(d_\Gamma(u,w_1), d_\Gamma(u,w_2), \ldots, d_\Gamma(u,w_k)\big) \neq r(v \mid W), \]
        where $W = \{w_1, w_2, \ldots, w_k\}$ and { $d_\Gamma(u,w_i)$ denotes the distance between vertices $u$ and $w_i$ in $\Gamma$,  for  $1\le i \le k$. A resolving set of $\Gamma$ with the smallest cardinality is called a \emph{metric basis} of $\Gamma$. The cardinality of a metric basis of $\Gamma$ is the \emph{metric dimension} of $\Gamma$, and is denoted by $\dim(\Gamma)$}~\cite{manuel2016total}. 
        
        The metric dimension was introduced independently by Slater in 1975 in the study of optimal sensor placement in networks and by Harary and Melter in 1976 in a combinatorial context. Since then, it has been extensively investigated and generalized to numerous graph classes, with notable applications in autonomous robot navigation, network discovery, and chemical graph theory.

		\begin{lem}
			\label{lem:if-G-has-one-nongenerator-then-dim-Gamma-G-equal-n-1}
			{Let $G$ be a nontrivial cyclic group of order $n$, and $S$ the set of all generators of $G$.} If $|\overline{S}| = n- |S| =1$, then $\dim(\Gamma(G)) =n-1$.
		\end{lem}
		
		\begin{proof}
			Suppose that $|\overline{S}| = n - |S| =1$, where
			\begin{align*}
				S = \bigcup\limits_{i=1}^{|S|} \{g_i\} \quad \text{ and } \quad 	\overline{S} = \bigcup\limits_{i=1}^{|\overline{S}|} \{g_i^*\}
			\end{align*}
			arranged in the increasing order of indices. Let $W = S$. Then
			\begin{align*}
				r(g_1^*|W) = (\underbrace{1,\ldots,1}_{|W| }),
			\end{align*}
			so $W$ is a resolving set of $\Gamma(G)$ of cardinality $|W| = |S| =n-1$. Thus, $\dim(\Gamma(G)) \le n-1$.
			
			Let $W^*=S\setminus \{g_{1}\}$. Thus, $g_{1} \in S$ and $g_{1} \notin W^*$, so $|W^*| = n-2$. We show that $W^*$ is not a resolving set of $\Gamma(G)$. Note that $\overline{W^*} = \{g_1^*, g_{1}\}$, and
			\begin{align*}
				r(g_1^*|W^*)= (\underbrace{1,\ldots,1}_{|W^*| })=	r(g_{1}|W^*),
			\end{align*}
			implying that $W^*$ is not a resolving set of $\Gamma(G)$. Thus, $\dim(G) \ge n-1$, and consequently, $\dim(\Gamma(G)) =n-1$.
		\end{proof}

		\begin{thm}
			\label{thm:the-metric-dimension-of-the-generator-graph}
			Let $G$ be a nontrivial {cyclic} group of order $n$, and $S$ the set of all generators of $G$. If $\Gamma(G)$ is the generator graph of $G$, then the metric dimension of $\Gamma(G)$ is given by
			\begin{align*}
				\dim(\Gamma(G)) = \begin{cases}
					n-1 &\text{ if } n = |S| +1,\\
					n-2 & \text{ if } n \ge |S| +2.
				\end{cases}
			\end{align*}
		\end{thm}
		
		\begin{proof}
			Let $\Gamma(G)$ be the generator graph of a nontrivial cyclic group $G$ of order $n$, whose set of generators is denoted by $S$. If $n = |S|+1$, then $n-|S| = |\overline{S}| =1$ and by Lemma \ref{lem:if-G-has-one-nongenerator-then-dim-Gamma-G-equal-n-1}, $\dim(\Gamma(G)) = n-1$. Hence, we may assume that $n \ge |S|+2$. If $n \ge |S|+2$, then $n-|S|=|\overline{S}|\ge 2$. Let
			\begin{align*}
				S = \bigcup\limits_{i=1}^{|S|} \{g_i\} \quad \text{ and } \quad 	\overline{S} = \bigcup\limits_{j=1}^{|\overline{S}|} \{g_j^*\},
			\end{align*}
			and let
			\begin{align*}
				W = \left( \bigcup\limits_{i=2}^{|S|} \{g_i\}  \right) \cup \left(\bigcup\limits_{j=2}^{|\overline{S}|} \{g_j^*\} \right).
			\end{align*}
			Thus, $W = (S \setminus \{g_1\}) \cup (\overline{S} \setminus \{g_1^*\}) = (S\cup \overline{S}) \setminus \{g_1, g_1^*\} = V(\Gamma(G)) \setminus \{g_1, g_1^*\}$. Observe that
			\begin{align*}
				r(g_1|W) = (\underbrace{1,\ldots, 1}_{|W|}) \ne (\underbrace{1,\ldots, 1}_{|S|-1}, \overbrace{2,\ldots, 2}^{|\overline{S}|-1} ) = r(g_1^*|W).
			\end{align*}
			Hence, $W$ is a resolving set $\Gamma(G)$. Thus, $\dim(\Gamma(G)) \le |W| = n-2$.
			
			For an integer $i$, where $2\le i \le |S|$, and an integer $j$, where $2\le j \le |\overline{S}|$, suppose that $W' = V(\Gamma(G)) \setminus \{g_1, g_1^*, g_i\}$ or $W' = V(\Gamma(G)) \setminus \{g_1, g_1^*, g_j^*\}$. If $W' = V(\Gamma(G)) \setminus \{g_1, g_1^*, g_i\}$, then 
			\begin{align*}
				r(g_1|W') = (\underbrace{1,\ldots, 1}_{|W'|}) = r(g_i|W'),
			\end{align*}
			and so $W'$ is not a resolving set of $\Gamma(G)$. On the other hand, if $W' = V(\Gamma(G)) \setminus \{g_1, g_1^*, g_j^*\}$, then
			\begin{align*}
				r(g_1^*|W') = (\underbrace{1,\ldots, 1}_{|S|-1}, \overbrace{2,\ldots, 2}^{|\overline{S}|-2} ) = r(g_j^*|W'),
			\end{align*}
			and again, $W'$ is not a resolving set of $\Gamma(G)$. Thus $\dim(\Gamma(G)) \ge |W| = n-2$. Consequently, $W$ is a metric basis of $\Gamma(G)$, and so $\dim(\Gamma(G)) = |W| = n-2$.
		\end{proof}

      The following result follows from Theorem \ref{thm:the-metric-dimension-of-the-generator-graph} and Theorem \ref{thm:group-of-prime-order-in-gg}.

        \begin{cor}
		\label{cor:the-metric-dimension-of-the-generator-graph-in-terms-of n}
		Let $G$ be a nontrivial cyclic group of order $n$, and $\Gamma(G)$  the generator graph of $G$. Then the metric dimension of $\Gamma(G)$ is given by
			\begin{align*}
				\dim(\Gamma(G)) = \begin{cases}
					n-1 &\text{ if } n \text{ is prime},\\
					n-2 & \text { otherwise}.
				\end{cases}
			\end{align*}
	   \end{cor}

	\newpage
	\bibliographystyle{abbrv} 
	\bibliography{references}

\begin{thebibliography}{10}

\bibitem{erdem2018generating}
F.~Erdem.
\newblock On the generating graphs of symmetric groups.
\newblock {\em Journal of Group Theory}, 21(4):629--649, 2018.

\bibitem{fraleigh2003first}
J.~B. Fraleigh.
\newblock {\em A First Course in Abstract Algebra}.
\newblock Addison-Wesley, Boston, MA, 7th edition, 2003.

\bibitem{gutman2021geometric}
I.~Gutman.
\newblock Geometric approach to degree-based topological indices: Sombor
  indices.
\newblock {\em MATCH Commun. Math. Comput. Chem}, 86(1):11--16, 2021.

\bibitem{henning2013total}
M.~A. Henning and A.~Yeo.
\newblock {\em Total domination in graphs}.
\newblock Springer, 2013.

\bibitem{lucchini2017generating}
A.~Lucchini, A.~Mar{\'o}ti, and C.~M. Roney-Dougal.
\newblock On the generating graph of a simple group.
\newblock {\em Journal of the Australian Mathematical Society}, 103(1):91--103,
  2017.

\bibitem{manuel2016total}
P.~Manuel, M.~Binu, B.~Rajan, and S.~Stephen.
\newblock On the metric dimension of the total graph of a graph.
\newblock {\em Discrete Mathematics, Algorithms and Applications},
  8(01):1650017, 2016.

\bibitem{Mazorodze2014GutmanIndex}
J.~P. Mazorodze, S.~Mukwembi, and T.~Vetr{\'i}k.
\newblock On the {G}utman index and minimum degree.
\newblock {\em Discrete Applied Mathematics}, 173:77--82, 2014.

\bibitem{randic1975characterization}
M.~Randi{\'c}.
\newblock Characterization of molecular branching.
\newblock {\em Journal of the American Chemical Society}, 97(23):6609--6615,
  1975.

\bibitem{tacbobo2023generator}
T.~L. Tacbobo.
\newblock The generator graph of a group.
\newblock {\em European Journal of Pure and Applied Mathematics},
  16(3):1894--1901, 2023.

\bibitem{vasantha2009groups}
W.~Vasantha~Kandasamy and F.~Smarandache.
\newblock Groups as graphs, editura cuart.
\newblock {\em Romania}, 2009.

\bibitem{wiener1947structural}
H.~Wiener.
\newblock Structural determination of paraffin boiling points.
\newblock {\em Journal of the American chemical society}, 69(1):17--20, 1947.

\bibitem{Zhong2012Harmonic}
L.~Zhong.
\newblock The harmonic index for graphs.
\newblock {\em Applied Mathematics Letters}, 25(3):561--566, 2012.

\end{thebibliography}

\end{document}